\documentclass[a4paper]{article}
\usepackage{amsmath}
\usepackage{graphicx}
\usepackage{faktor}
\usepackage{amsmath}
\usepackage{amsthm}
\usepackage{mathrsfs}
\usepackage{latexsym}
\usepackage{amssymb}
\usepackage{amscd}
\usepackage[all]{xy}
\usepackage{authblk}
\usepackage{hyperref}

\title{The path algebra as a left adjoint functor}

\author[1]{Kostiantyn Iusenko}
\author[2]{John William MacQuarrie}
\affil[1]{Instituto de Matem\'{a}tica e Estat\'{i}stica, USP, Brazil,\newline {\tt \href{mailto:iusenko@ime.usp.br}{iusenko@ime.usp.br}}}
\affil[2]{Universidade Federal de Minas Gerais, Brazil
, \newline {\tt \href{mailto:john@mat.ufmg.br}{john@mat.ufmg.br}}}

\DeclareSymbolFont{AMSb}{U}{msb}{m}{n}
\DeclareMathSymbol{\N}{\mathbin}{AMSb}{"4E}
\DeclareMathSymbol{\Z}{\mathbin}{AMSb}{"5A}
\DeclareMathSymbol{\R}{\mathbin}{AMSb}{"52}
\DeclareMathSymbol{\Q}{\mathbin}{AMSb}{"51}
\DeclareMathSymbol{\I}{\mathbin}{AMSb}{"49}
\DeclareMathSymbol{\C}{\mathbin}{AMSb}{"43}

\newcommand{\dbl}{[\hspace{-0.2ex}[}
\newcommand{\dbr}{]\hspace{-0.2ex}]}
\newcommand{\db}[1]{\dbl {#1} \dbr}

\newcommand{\ctens}{\widehat{\otimes}}

\newcommand{\iso}{\cong}
\newcommand{\invlim}{\underleftarrow{\textnormal{lim}}\,}

\newcommand{\onto}{\twoheadrightarrow}

\newcommand{\id}{\textnormal{id}}

\newcommand{\Hom}{\textnormal{Hom}}

\newcommand{\tn}[1]{\textnormal{#1}}
\newcommand{\cat}[1]{\tn{\textbf{#1}}}

\newcommand{\ospeech}{\textquotedblleft}
\newcommand{\cspeech}{\textquotedblright}

\numberwithin{equation}{section}

\begin{document}

\newtheorem{defn}[equation]{Def{i}nition}
\newtheorem{prop}[equation]{Proposition}
\newtheorem{lemma}[equation]{Lemma}
\newtheorem{theorem}[equation]{Theorem}
\newtheorem{corol}[equation]{Corollary}
\newtheorem{exercise}[equation]{Exercise}
\newtheorem{remark}[equation]{Remark}

\theoremstyle{definition}
\newtheorem{example}[equation]{Examples}

\maketitle

\begin{abstract}
We consider an intermediate category between the category of finite quivers and a certain category of pseudocompact associative algebras whose objects include all pointed finite dimensional algebras.  We define the completed path algebra and the Gabriel quiver as functors.  We give an explicit quotient of the category of algebras on which these functors form an adjoint pair.  We show that these functors respect ideals, obtaining in this way an equivalence between related categories.
\end{abstract}

\section{Introduction}

A remarkably simple and ingenious construction of Gabriel \cite{Gabriel_thesis} showed that the representation theory of finite dimensional associative algebras could be treated combinatorially.  This approach revolutionized the subject and remains one of the main tools when working in the area.  Without going into details (we do so later), the idea of the construction is as follows: given a finite dimensional pointed associative algebra $A$, one constructs a finite directed graph $Q_A$ (the \emph{ Gabriel quiver} of $A$).  From $Q_A$, one obtains an associative algebra, called the \emph{path algebra} of $Q_A$, having $A$ as a quotient.  Gabriel observed that a great deal of the representation theory of $A$ can be treated in terms of the path algebra.

When we construct the (completed) path algebra $B$ of a quiver $Q$, the vertices of $Q$ correspond to $B/J(B)$ (where $J(B)$ denotes the Jacobson radical of $B$) and the arrows correspond to $J(B)/J^2(B)$, while the deeper radical layers are \ospeech extended freely\cspeech, in the sense that they are as large as possible given the constraints imposed by $Q$.  As the path algebra is a free construction, it is natural to suppose that it should be a left adjoint.  What is more, there is an obvious candidate for the corresponding right adjoint: the Gabriel quiver construction.

We show in this article that this reasonably intuitive idea is correct, with some care.  For example, the arrows of the Gabriel quiver of a given algebra $A$ are given by a choice of basis for certain subspaces of $J(A)/J^2(A)$ -- a choice that kills any hope of the Gabriel quiver construction being functorial.  We avoid this problem by simply replacing the arrows of a quiver by vector spaces, thus avoiding the choice of a basis.  Other, more subtle choices involved in the Gabriel quiver construction are treated by considering orbits under a certain group action and by
defining an explicit quotient of the category of algebras.  The adjunction we obtain thus explains in a very precise way what information is transferred in the conversation between quivers and algebras.

Our main interest is in the study of finite dimensional associative algebras.  However, finite dimensional algebras and finite quivers do not quite match up, since the path algebra of a quiver with loops or cycles is not finite dimensional.  For this reason, it is convenient to work with a larger class of algebras, namely the class of those pseudocompact algebras $A$ such that $A/J^2(A)$ is finite dimensional.

\medskip

The paper is organized as follows.  In Section \ref{section prelims} we present general results that we will require, many of which are well known.  In Section \ref{section categories} we introduce the categories of interest to us.  On the quiver side, we replace normal quivers with the category $\cat{VQuiv}$ of what we call Vquivers, essentially replacing the set of arrows between vertices with vector spaces.  On the algebra side, we define a certain quotient $\cat{PAlg}_1$ of our category of pointed algebras, which we believe warrants further study.  In Section \ref{section functors} we introduce several functors between the categories defined in Section \ref{section categories} -- the completed path algebra and a functorial version of the Gabriel quiver construction.  In Section \ref{section main adjunction} we prove our first main theorem: that the completed path algebra functor is left adjoint to the Gabriel quiver functor, both treated as functors between $\cat{VQuiv}$ and $\cat{PAlg}_1$.  In Section \ref{section related adjunctions} we prove two related adjunctions.  Firstly, thinking of the adjunction of Section \ref{section main adjunction} as giving a hereditary approximation of a given algebra, we show an easier adjunction that could be thought of as a semisimple approximation to an algebra. Secondly, we demonstrate that the Gabriel quiver functor also has a right adjoint, defining it explicitly.  Finally, in Section \ref{section adjoint equivalence} we show how ideals can be brought into the picture, obtaining in this way an equivalence of categories.

\medskip

We mention two motivations for this work.  Firstly, the adjunction of Section \ref{section main adjunction} provides a new tool in the study of finite dimensional associative algebras -- one may now attack a given problem using formal properties of adjoint functors.  The second motivation is generalization.  Pseudocompact algebras appear as completed group algebras of profinite groups, and hence are studied in Galois theory, algebraic number theory, algebraic geometry and profinite group theory.  Functorial definitions of foundational objects of the theory of finite dimensional associative algebras will allow both definitions and theorems to extend naturally to this wider class of algebras.  In the future, we will use the results of this article to further develop a combinatorial approach to the representation theory of arbitrary pseudocompact algebras.

\section*{Acknowledgements} The first author was partially
supported by Fapesp grants 2014/09310-5, 2015/00116-4 and by CNPq grant 456698/2014-0. The second author was  partially supported by CNPq grant 443387/2014-1 and by FAPEMIG grant PPM-00481-16.

% \medskip

% \textbf{Acknowledgements} The first author was partially
% supported by Fapesp grants 2014/09310-5, 2015/00116-4 and by CNPq grant 456698/2014-0. The second author was  partially supported by CNPq grant 443387/2014-1 and by FAPEMIG grant PPM-00481-16.

\section{Preliminaries}\label{section prelims}

We collect here well-known results and constructions in associative algebras that we will require later.  Throughout this article, $k$ will denote a perfect field treated as a discrete topological ring.

\begin{defn}[\cite{Brumer}]
A \emph{pseudocompact algebra} is an associative, unital, Hausdorff topological $k$-algebra $A$ possessing a basis of 0 consisting of open ideals $I$ having cofinite dimension in $A$ that intersect in 0 and such that $A \iso \invlim_I A/I$.

\medskip

Let $A,B$ be pseudocompact algebras.  A \emph{pseudocompact $A$-($B$-bi)module} is a topological $A$-($B$-bi)module $U$ possessing a basis of 0 consisting of open sub(bi)modules $V$ of finite codimension that intersect in 0 and such that $U \iso \invlim_V U/V$.
\end{defn}

The category of pseudocompact modules for a pseudocompact algebra has exact inverse limits \cite[\S 1]{Brumer}.

\begin{lemma}\label{linearly compact facts}
Let $A$ be a pseudocompact algebra and $V,U$ pseudocompact $A$-modules.
\begin{enumerate}
\item The module $V$ is \emph{linearly compact}, meaning that if ever we have a collection of cosets of closed subspaces $\{V_i\}$ of $V$ with the finite intersection property, then $\bigcap V_i\neq \varnothing$.

\item\label{linearly compact facts part 2} If ever $\rho:V\to U$ is a continuous homomorphism, then $\rho(V)$ is linearly compact and hence closed in $U$.

\item The submodule abstractly generated by a finite subset of $V$ is closed.
\end{enumerate}
\end{lemma}

\begin{proof}
\begin{enumerate}
\item A finite dimensional $A$-module is linearly compact \cite[II.27.7]{Lefschetz} and hence $V$ is linearly compact by \cite[Proposition 4]{Zelinsky}.

 \item This is \cite[II.27.4, II.27.5]{Lefschetz}.
\item By \cite[Lemma 1.2]{Brumer} any finitely generated submodule of $V$ is the homomorphic image of a continuous homomorphism $\rho:A^n\to V$ of the free $A$-module $A^n$ (some $n\in \N$) and hence is closed by Part \ref{linearly compact facts part 2}.
\end{enumerate}
\end{proof}

The \textit{Jacobson radical} $J(A)$ of a pseudocompact algebra $A$ is the intersection of the maximal open left ideals of $A$.  We have that $J(A)$ is equal to the intersection of the maximal open right ideals of $A$, and (see \cite[\S 1]{Brumer}) to the intersection of the maximal open two-sided ideals of $A$.

\begin{lemma}\label{radical limit of radicals}
Write $A = \invlim_{I} \{A/I\,,\,\alpha_{II'}:A/I'\to A/I\}$ as an inverse limit of finite dimensional quotient algebras $A/I$.  Then
\[J(A) = \invlim_I J(A/I).\]
\end{lemma}

\begin{proof}
Since the quotients $\alpha_{II'}$ are surjective, it is immediate that $\alpha_{II'}(J(A/I'))\subseteq J(A/I)$.  Thus, the restriction of the inverse system of $A/I$ to their Jacobson radicals indeed yields an inverse system with inverse limit $\invlim J(A/I)\subseteq A$.  Since an element $x\in J(A)$ maps into each $J(A/I)$, it follows that $J(A)\subseteq \invlim J(A/I)$.  On the other hand, given $x\not\in J(A)$, there is some open maximal left ideal $M$ not containing $x$.  Working within the cofinal subsystem of $A/I$ with $I\subseteq M$, we see that $x+I\not\in J(A/I)$, and hence $x\not\in\invlim J(A/I)$, so that $\invlim J(A/I)\subseteq J(A)$.
\end{proof}

\begin{corol}\label{radical surjective}
Let $A,B$ be pseudocompact algebras and let $\alpha:A\to B$ be a continuous surjective algebra homomorphism.  Then $\alpha(J(A)) = J(B)$. 
\end{corol}

\begin{proof}
That $\alpha(J(A))\subseteq J(B)$ is immediate.  To prove the other inclusion, we begin by supposing that $A, B$ are finite dimensional.  Recall \cite[Proposition 3.5]{ARS} that the radical $\tn{Rad}_B(U)$ of a finitely generated $B$-module $U$ is given by $J(B)U$ (and similarly for $A$).  We can treat $B$ either as a $B$-module, or as an $A$-module via $\alpha$, and since $\alpha$ is surjective it follows that $\tn{Rad}_BB = \tn{Rad}_AB$.  Thus
\[
	J(B) = J(B)B = \tn{Rad}_B B = \tn{Rad}_A B = J(A)\cdot B = \alpha(J(A)) B \subseteq \alpha(J(A)). 
\]

Now let $A$ be general and $B$ finite dimensional.  Since $B$ is discrete and $\alpha$ is continuous, we can consider a cofinal subset of open ideals of $A$ contained in the kernel of $\alpha$.  By factorizing $\alpha$ through these quotients we obtain a map of inverse systems $\{\alpha_I:A/I\to B\,|\,I\lhd_O A, I\leqslant \tn{Ker}(\alpha)\}$.  Restricting and corestricting this inverse system to the Jacobson radicals, we obtain a surjective map of inverse systems $\alpha_I:J(A/I)\onto J(B)$, whose inverse limit is $\alpha:J(A)\to J(B)$ by Lemma \ref{radical limit of radicals}.  It is onto by the exactness of $\invlim$.

Finally, allowing both $A$ and $B = \invlim\{B/K,\beta_{KK'}\}$ to be general, the obvious composition $\beta_K\alpha:A\to B\to B/K$ is a surjective map onto the finite dimensional algebra $B/K$ and hence it restricts to a surjection $J(A)\onto J(B/K)$ for each $K$.  We obtain in this way a surjective map of inverse systems and the result follows from the exactness of $\invlim$.
\end{proof}

Define $J^0(A)$ to be $A$ and $J^1(A) = J(A)$.  For each $n>1$, define $J^n(A)$ to be $J(A)J^{n-1}(A)$.  When $A$ is understood, we will denote these ideals by $J^n$, rather than $J^n(A)$.

\begin{lemma}\label{radical facts when J by J2 is fd}
Let $A$ be a pseudocompact algebra with $A/J^2$ finite dimensional.
\begin{enumerate}
\item The Jacobson radical $J$ is finitely generated as an $A$-module.
\item If $U$ is a finitely generated pseudocompact $A$-module, then $JU$ is a closed submodule of $U$.
\item $J^n$ is a closed, finitely generated submodule of $A$ for each $n\in \N$.
\end{enumerate}
\end{lemma}

\begin{proof}
\begin{enumerate}
\item The module $J/J^2$ is finite dimensional, hence finitely generated.  Thus by \cite[Corollary 1.5]{Brumer}, $J$ is finitely generated.

\item We have that $JU$ is generated by elements of the form  $xu$, where $x$ runs through a finite generating set for $J$ (which exists by Part 1) and $u$ runs through a finite generating set of $U$.  Hence $JU$ is finitely generated, thus closed in $U$ by Lemma \ref{linearly compact facts}.

\item Both $J^0 = A$ and $J^1 = J$ are closed, finitely generated submodules of $A$.  Suppose that some $J^m$ is finitely generated and closed.  Then $J^{m+1} = JJ^m$ is finitely generated, and hence closed by Lemma \ref{linearly compact facts}.
\end{enumerate}
\end{proof}
Working by induction and observing that we have surjective maps from $J/J^2\times J^{n}/J^{n+1}$ to $J^{n+1}/J^{n+2}$, one sees that $A/J^n$ is finite dimensional, so that each $J^n$ is in fact open in $A$.

\begin{lemma}\label{radical powers surjective}
Let $A,B$ be pseudocompact algebras with $A/J^2(A)$ and $B/J^2(B)$ finite dimensional and let $\alpha:A\to B$ be a continuous surjective algebra homomorphism.  Then $\alpha(J^n(A)) = J^n(B)$ for each $n\geqslant 0$. 
\end{lemma}

\begin{proof}
This follows by induction using Lemma \ref{radical facts when J by J2 is fd}, with Corollary \ref{radical surjective} as base case.
\end{proof}

\begin{prop}\label{prop complete wrt J}
Let $A$ be a pseudocompact algebra with $A/J^2$ of finite dimension.  Then $A$ is complete with respect to the $J$-adic topology.  That is,
\[A = \invlim_{n\in \N}A/J^n.\]
\end{prop}

\begin{proof}
The quotients $A\onto A/J^n$ yield a continuous surjection $A\to \invlim A/J^n$ by exactness of $\invlim$.  Since everybody is pseudocompact, the kernel is $\bigcap J^n = 0$, by \cite[Lemma 5.1(c)]{SimsonCoalg}.
\end{proof}

We state a generalization due to Curtis of the Wedderburn-Malcev Theorem to the algebras that interest us:

\begin{prop}\label{WedderburnMalcev}
Let $A$ be a pseudocompact algebra such that $A/J^2$ is finite dimensional.  There is a closed semisimple subalgebra $\Sigma$ of $A$ having the property that $A = \Sigma\oplus J(A)$ as a vector space.  If $s,t$ are two splittings of $A\to A/J$ as an algebra, then there exists an element $w\in J(A)$ such that
\[s(z) = {}^{1+w}t(z) := (1+w)t(z)(1+w)^{-1}\]
for any $z\in A/J$.
\end{prop}

\begin{proof}
In light of results above, this follows directly from \cite[Theorem 1]{Curtis}.  The uniqueness conditions in \cite{Curtis} are equivalent to ours.
\end{proof}

A splitting of $A\onto A/J$ allows us to treat $A/J$ as a subalgebra $\Sigma$ of $A$ (which depends on the splitting).  In this way, we can regard any $A$-(bi)module as a (bi)module for the semisimple algebra $\Sigma$ by restricting coefficients.  In particular, the pseudocompact module $J^n$ is a $\Sigma$-bimodule for each $n$.

\begin{lemma}\label{Lemma J to J by J2 splits}
Let $A = \Sigma\oplus J$ be a pseudocompact algebra such that $A/J^2$ has finite dimension, decomposed as in Proposition \ref{WedderburnMalcev}.  The $\Sigma$-bimodule homomorphism $J\onto J/J^2$ splits.
\end{lemma}

\begin{proof}
Recall that an $A/J$-bimodule is the same thing as a left $A/J\otimes (A/J)^{op}$-module.  But $k$ is perfect, so by \cite[Theorem 6.4]{Lang} this algebra is semisimple.  In particular, the bimodule $J/J^2$ is projective in the pseudocompact category and the map $J\onto J/J^2$ splits. 
\end{proof}

The following definition, which is fundamental to us, makes use of the \emph{completed tensor product}, for which see \cite[\S 7.5]{Gabriel2}.  Because it causes no further complications, we state the definitions in greater generality than we require, but note that in this article we only take tensor products of finite dimensional modules, and so the completed tensor product coincides with the normal tensor product.  
   
 \begin{defn}[Completed tensor algebra]
Let $\Sigma$ be a pseudocompact algebra and $V$ a pseudocompact $\Sigma$-bimodule.  The \emph{completed tensor algebra} $T\db{\Sigma,V}$ is defined to be
\[T\db{\Sigma,V} = \prod_{n=0}^{\infty}V^{\ctens_n},\]
where $V^{\ctens_0} = \Sigma, V^{\ctens_1} = V$ and $V^{\ctens_n} = \underbrace{V\ctens_{\Sigma} V\ctens_{\Sigma}\hdots\ctens_{\Sigma}V}_{n-\mbox{times}}$ for $n>1$.

Multiplication is given in the obvious way: the product of the pure tensors $v_1\ctens\hdots\ctens v_m\in V^{\ctens_m}$ and $w_1\ctens\hdots\ctens w_n\in V^{\ctens_n}$ is 
\[v_1\ctens\hdots \ctens v_m\ctens w_1\ctens\hdots\ctens w_n\in V^{\ctens_{m+n}}.\]
\end{defn}

Observe that the completed tensor algebra is again a pseudocompact algebra \cite[\S 7.5]{Gabriel2}. The ideals $T\db{\Sigma,V}_{\geqslant s} = \prod_{n\geqslant s}V^{\ctens_n}$ are clearly closed in $T\db{\Sigma,V}$.
Furthermore, $T\db{\Sigma,V} = \invlim_{s\in \N}T\db{\Sigma,V}/T\db{\Sigma,V}_{\geqslant s}$.

\begin{example}
\begin{enumerate}
\item Suppose that $\Sigma,V$ are finite dimensional and that for some $n\in \N$ we have $V^{\ctens_n}=0$.  Then $T\db{\Sigma,V}$ is finite dimensional and coincides with the classical tensor algebra $T(\Sigma,V)$ (for which, see for instance \cite[\S III.1]{ARS}).

\item Let $\Sigma = k$ and let $V=k$ treated as a  $\Sigma$-bimodule in the obvious way.  Then $T(\Sigma,V)$ is isomorphic to the polynomial ring in one variable.  On the other hand, the corresponding completed tensor algebra is isomorphic to the ring $T\db{\Sigma,V} \iso k\db{x}$ of formal power series in the variable $x$.
\end{enumerate}
\end{example}

\begin{lemma}\label{Lemma radical of CPA what we expect}
Let $\Sigma$ be a finite dimensional commutative semisimple $k$-algebra and $V$ a finite dimensional $\Sigma$-bimodule.  Then
\[J\left(T\db{\Sigma,V}\right) = \prod_{n=1}^{\infty}V^{\ctens_n}.\]
\end{lemma}

\begin{proof}
Denote by $M$ the ideal $\prod_{n=1}^{\infty}V^{\ctens_n}$ and by $J$ the ideal $J\left(T\db{\Sigma,V}\right)$.  If $I$ is a maximal left ideal of $\Sigma$ then $I+M$ is a maximal left ideal of $T\db{\Sigma,V}$ and so, since $\Sigma$ is semisimple, $J\subseteq M$.  On the other hand, given $m\in M$, evaluation at $m$ yields by \cite[Chapter 4, \S 4, Proposition 4]{bourbaki} a ring homomorphism from the power series ring $\Sigma\db{x}$ to $T\db{\Sigma,V}$, because $T\db{\Sigma,V}$ is complete with respect to $M$.  It follows that $1+m$ is invertible, so that $m\in J$.
\end{proof}

The completed tensor algebra is given by the following universal property:

\begin{lemma}\label{Lemma UP tensor algebra}
Let $\Sigma$ be a pseudocompact algebra, $V$ a pseudocompact $\Sigma$-bimodule, and $A$ a pseudocompact algebra.  Given a continuous algebra homomorphism
\[\alpha_0:\Sigma\to A\]
and a continuous $\Sigma$-bimodule homomorphism
\[\alpha_1:V\to A\]
with $A$ treated as a $\Sigma$-bimodule via $\alpha_0$, there exists a unique continuous algebra homomorphism
\[\alpha:T\db{\Sigma,V}\to A\]
such that $\alpha|_{\Sigma} = \alpha_0$ and $\alpha|_{V} = \alpha_1$. 
\end{lemma}

\begin{proof}
Just as in the proof of the abstract version of this result \cite[Lemma III.1.2]{ARS}, we obtain from $\alpha_1$ continuous $\Sigma$-bimodule homomorphisms
\[\alpha_n:V^{\ctens_n}\to A\]
in the obvious way for each $n>1$.  By summing these maps, we obtain continuous $\Sigma$-bimodule homomorphisms $\alpha_{\leqslant m}:\prod_{n=0}^mV^{\ctens_n}\to A$, which together yield a map of inverse systems of $\Sigma$-bimodules, and hence a continuous $\Sigma$-bimodule homomorphism
\[\alpha:T\db{\Sigma,V} = \invlim_{m}\prod_{n=0}^mV^{\ctens_n}\to A.\]
We need only check that $\alpha$ is a homomorphism of algebras.  Write $A$ as the inverse limit of finite dimensional quotients $A_i$.  The induced map $T\db{\Sigma,V}\to A_i$ factors through some quotient $T\db{\Sigma,V}/T\db{\Sigma,V}_{\geqslant m}$, on which it is easily seen to be an algebra homomorphism.  Thus $\alpha$ is an inverse limit of algebra homomorphisms, and hence an algebra homomorphism.
\end{proof}

\section{Categories}\label{section categories}

\subsection*{Quivers and Vquivers}

We define the categories that will interest us.

\begin{defn}[eg. {\cite[\S 3.1]{ARS}}]
A \emph{finite quiver} $Q$ is a directed graph $(Q_0,Q_1)$ consisisting of a finite set $Q_0$ of vertices and a finite set $Q_1$ of arrows.  Given an arrow $a \in Q_1$, we denote by $\mathrm{source}(a)\in Q_0$ the source of $a$ and by $\mathrm{target}(a)\in Q_0$ its target.
\end{defn}

We can thus visualize an arrow in $Q$ as follows:
\[\xymatrix{
\mathrm{source}(a) \ar[r]^{a} & \mathrm{target}(a)}\]
Note that we permit directed cycles and loops.  A quiver having neither loops nor cycles is \emph{acyclic}.  A simple example of an acyclic quiver $Q$ is
\[\xymatrix{
1  \ar@/^/[r]^{a} \ar@/_/[r]_{b} & 2\ar[r]^{c} & 3}\]

\begin{defn}
A \emph{map of quivers} $\rho:(Q_0,Q_1)\to (R_0,R_1)$ from the finite quiver $Q = (Q_0,Q_1)$ to the finite quiver $(R_0,R_1)$ is an injective map $\rho_0:Q_0\to R_0$ together with a map $\rho_1:Q_1\to R_1$ such that $\rho(\mathrm{source}(a)) = \mathrm{source}(\rho(a))$ and  $\rho(\mathrm{target}(a)) = \mathrm{target}(\rho(a))$ for every $a\in Q_1$.  The map $\rho$ is said to be \emph{injective} if $\rho_1$ is injective.
\end{defn}

An injective map of quivers could be thought of as the inclusion of a smaller quiver into a larger quiver.

\begin{defn}
The category $\cat{Quiv}$ has objects finite quivers and morphisms maps of quivers.  The subcategory of $\cat{Quiv}$ whose morphisms are injective quiver maps will be denoted $\cat{IQuiv}$.  
\end{defn}

For our purposes it is important to consider not a finite set of arrows between two vertices of a quiver, but instead a finite dimensional vector space.  A similar approach is taken in \cite{DWZ}.  We call such objects Vquivers.  Recall that we have fixed a perfect field $k$.

\begin{defn}\label{def Vquiver}
A (pointed) \emph{finite Vquiver} $VQ = (VQ_0^*, VQ_{e,f})$ consists of a finite set of vertices $VQ_0^* = \{*\}\cup VQ_0$ and, for each pair $e,f\in VQ_0^*$ a finite dimensional $k$-vector space $VQ_{e,f}$, subject to the condition that $VQ_{*,e} = VQ_{e,*} = 0$ for each vertex $e$.
\end{defn}

We call the vertex $*$ the \emph{point}.  The role of the point is to simplify the definition of maps of Vquivers.  One obtains a Vquiver from a quiver $Q$ by simply adding the point and replacing the set $Q_{e,f}$ of arrows from the vertex $e$ to the vertex $f$ by the vector space having basis $Q_{e,f}$.  The Vquiver corresponding to the quiver above is 
\[\xymatrix{
1  \ar[rr]^{k^2 = \langle a,b \rangle} && 2\ar[rr]^{k = \langle c \rangle} && 3 & {*} }\]
with all unmarked vector spaces being 0.  We hope that the reader will agree that Vquivers are no more complicated than normal quivers.  The condition of acyclicity corresponds to the demand that for any sequence $e_1,\hdots,e_n$ of vertices ($n\geqslant 1$), at least one of the $n$ vector spaces $VQ_{e_1,e_2},VQ_{e_2,e_3},\hdots,VQ_{e_n,e_1}$ has dimension 0.  We call such a Vquiver \emph{acyclic}.

\medskip

The obvious map of Vquivers $VQ\to VR$ induced by the map of quivers $Q\to R$ does not correspond well to what happens at the level of algebras (it corresponds much better to what happens at the level of \emph{co}algebras -- more on this will be published at a later time).

\begin{defn}
A \emph{map of Vquivers} $\rho:(VR_0^*,VR_{e,f}) \to (VQ_0^*,VQ_{e',f'})$ from the finite Vquiver $(VR_0^*,VR_{e,f})$ to the finite Vquiver $(VQ_0^*,VQ_{e',f'})$ consists of 
\begin{itemize}
\item a pointed map $\rho_0:VR_0^*\to VQ_0^*$ (that is, such that $\rho_0(*)=*$) that restricts to a bijection from the elements of $VQ_0$ not mapping to $*$ onto $VR_0$.
\item a linear map $\rho_{e,f}:VR_{e,f}\to VQ_{\rho_0(e),\rho_0(f)}$ for each pair of vertices $e,f\in VR_0^*$.
\end{itemize}
We say that $\rho$ is \emph{surjective} if every $\rho_{e,f}$ is surjective.
\end{defn}

We note in passing that the map $\rho_0$ may also be described as a surjective map of vector spaces over \ospeech the field with one element\cspeech, for which see for instance \cite[Definition 2.1]{szczesny}.

\begin{example}
Let $VR$ be the Vquiver
\[\xymatrix{
1  \ar[rr]^{k^2} && 2\ar[rr]^{k} && 3 & {*} }\]
and let $VQ$ be the Vquiver
\[\xymatrix{
4\ar[rr]^{k^2} && 5 & {*} }\]

The map on vertices $1\mapsto 4, 2\mapsto 5, 3\mapsto *$ can be extended to a surjective map of Vquivers via any surjective linear map 
\[VR_{1,2} = k^2\to k^2 = VQ_{4,5}.\]
The vertex map $1\mapsto *, 2\mapsto 4, 3\mapsto 5$ also yields maps of Vquivers, but they are not surjective.  Any other choice of vertex map $h\mapsto 4, i\mapsto 5, j\mapsto *$ yields a \ospeech zero map\cspeech of Vquivers.
\end{example}

\begin{defn}
The category $\cat{VQuiv}$ has objects finite Vquivers and morphisms maps of Vquivers.  The subcategory of $\cat{VQuiv}$ whose morphisms are surjective Vquiver maps will be denoted $\cat{SVQuiv}$.  
\end{defn}

\subsection*{Algebras}

A pseudocompact algebra $A$ is said to be \emph{basic} if $A/J(A)$ is isomorphic to a product of $k$-division algebras.  It is \emph{pointed} if $A/J(A)$ is isomorphic to a product of copies of the field $k$.  Morita's theorem for finite dimensional algebras has been extended to pseudocompact algebras (see \cite{Gabriel_thesis} or \cite[Proposition 5.6]{SimsonCoalg}).  Thus, for any pseudocompact algebra $A$, there is a basic pseudocompact algebra $B$ such that the categories of pseudocompact $A$-modules and of pseudocompact $B$-modules are equivalent.  That is to say, from the perspective of the representation theory of algebras, there is no loss in generality in supposing that $A$ is basic.  Furthermore, if $k$ is algebraically closed, then basic and pointed algebras coincide.  This justifies the common restriction in representation theory to the study of pointed algebras, which we will adopt here.

Observe that there is an unfortunate coincidence of standard terminology: the \ospeech pointed\cspeech of pointed sets (as in Definition \ref{def Vquiver}) has nothing to do with the \ospeech pointed\cspeech of pointed algebras.  We do not anticipate any confusion.

\medskip

The following two lemmas explain how we must define our category of algebras.  The inclusion of lower triangular $2\times 2$ matrices into $M_2(k)$ shows that a non-surjective algebra homomorphism $A\to B$ need not send $J(A)$ into $J(B)$.  But problems like this do not arise when $B$ is pointed (indeed, the following proof works when $B$ is basic):

\begin{lemma}\label{Lemma alg hom sends J to J}
Let $A,B$ be pointed pseudocompact algebras whose radical quotients are finite dimensional and let $\alpha:A\to B$ be a unital algebra homomorphism.  Then $\alpha(J(A))\subseteq J(B)$. 
\end{lemma}

\begin{proof}
Suppose first that $A$ and $B$ are finite dimensional.  Writing $B = \Sigma \oplus J(B)$ as in Proposition \ref{WedderburnMalcev}, $\Sigma$ is a product of copies of $k$, and hence the only nilpotent element of $\Sigma$ is 0.  It follows that an element of $B$ is in $J(B)$ if, and only if, it is nilpotent.  Given $w\in J(A)$, $\alpha(w)$ is nilpotent, so in $J(B)$.  The proof for pseudocompact algebras now mimics that of Corollary \ref{radical surjective}.
\end{proof}

Recall that a primitive idempotent of an algebra $A$ is an idempotent that cannot be written as the sum of two  non-zero orthogonal idempotents.  A homomorphism of algebras $\alpha$ is said to respect primitive idempotents if $\alpha(e)$ is a primitive idempotent whenever $e$ is a primitive idempotent.

\begin{lemma}
Let $\alpha:A\to B$ be a unital homomorphism of algebras.  Then $\alpha$ respects primitive idempotents if, and only if, the induced homomorphism $A/J(A) \to B/J(B)$ is surjective.
\end{lemma}

\begin{proof}
The induced map makes sense by Lemma \ref{Lemma alg hom sends J to J}.  Let $\{e_1,\hdots,e_n\}$ be a complete set of primitive orthogonal idempotents for $A$ and suppose that those not mapped to $0$ by $\alpha$ are $e_1,\hdots,e_s$.  Then $\{\alpha(e_1),\hdots,\alpha(e_s)\}$ is on one hand a set of orthogonal idempotents of $B$ summing to $1_B$ and on the other hand a basis of $\alpha(A/J(A))$ in $B/J(B)$.  The result follows.
\end{proof}

\begin{defn}
Denote by \cat{PAlg} the category whose objects are those pointed pseudocompact $k$-algebras $A$ with the property that $A/J^2(A)$ has finite dimension.  Morphisms in this category are continuous unital homomorphisms of algebras such that the induced map between the radical quotients is surjective.  Denote by $\cat{SPAlg}$ the subcategory of $\cat{PAlg}$ having the same objects as $\cat{PAlg}$ and morphisms all surjective unital algebra homomorphisms. 
\end{defn}

By way of examples, we mention that the category $\cat{PAlg}$ contains the non-surjective inclusion of diagonal $2\times 2$ matrices into lower triangular $2\times 2$ matrices, but it does not include the diagonal inclusion of $k$ into $k\times k$, because this map does not respect primitive idempotents.

\medskip

The categories of most interest to us in this discussion are quotients of $\cat{PAlg}$.  Thus we define relations on the morphisms in $\cat{PAlg}$:

\begin{defn}
Consider two objects $A,B$ and two morphisms $\alpha,\beta:A\to B$ in $\cat{PAlg}$.  
\begin{itemize}
\item[] Say that $\alpha \sim_0\beta$ if $(\alpha-\beta)(A)\subseteq J(B)$.
\item[] Say that $\alpha\sim_1\beta$ if $(\alpha-\beta)(A)\subseteq J(B)$ and $(\alpha-\beta)(J(A))\subseteq J^2(B)$.
\end{itemize}
\end{defn}

We regard two morphisms $\alpha,\beta$ with $\alpha\sim_i\beta$ ($i=0,1$) as being \ospeech close\cspeech, in so far as their difference is \ospeech small\cspeech.  The map $\alpha-\beta$ is not an algebra homomorphism, but this does not cause any difficulties.  Observe that if ever $\alpha\sim_1\beta$, then $\alpha\sim_0\beta$, and hence the relation $\sim_1$ is more refined that $\sim_0$.  One could define further relations $\sim_n$ recursively as follows: say that $\alpha\sim_n\beta$ ($n\geqslant 2$) if $\alpha\sim_{n-1}\beta$ and $(\alpha-\beta)(J^n(A))\subseteq J^{n+1}(B)$.  The following explains that by doing so we would get nothing new:   

\begin{lemma}\label{tilde one implies tilde n}
If $\alpha\sim_1\beta$ in $\cat{PAlg}$, then $\alpha\sim_n\beta$ for any $n\in \N$.
\end{lemma}

\begin{proof}
For any $j\in J(A)$ we have $\alpha(j)-\beta(j)\in J^2(B)$.  We check that $\alpha \sim_2 \beta$ (the rest follows by induction).  So consider $j_1,j_2\in J(A)$ and $j_1j_2\in J^2(A)$.  We have
\begin{align*}
    	(\alpha-\beta)(j_1j_2) & = \alpha(j_1)\alpha(j_2) - \beta(j_1)\beta(j_2) \\
        & = \alpha(j_1)\alpha(j_2)-\alpha(j_1)\beta(j_2) + \alpha(j_1)\beta(j_2) -\beta(j_1)\beta(j_2) \\
		& = \alpha(j_1)(\alpha(j_2)-\beta(j_2))+(\alpha(j_1)-\beta(j_1))\beta(j_2)
    \end{align*}
which is an element of $J(B)J^2(B) + J^2(B)J(B) = J^3(B)$.
\end{proof}

That said, one could indeed define finer relations on $\cat{PAlg}$.  For example, one could say that $\alpha \sim\beta$ when $(\alpha-\beta)(A)\subseteq J^n(B)$ for some fixed $n$ (cf. \cite[Definition 2.5]{DWZ}).  Clearly there are a vast array of relations of this sort to consider and we suggest that cleverly chosen relations may warrant investigation.

In order that $\sim_i$ defines a quotient category, we must check that these relations are congruences in the sense of \cite[\S II.8]{MacLane}.

\begin{prop} The relations $\sim_0, \sim_1$ are congruence relations on $\cat{PAlg}$ and $\cat{SPAlg}$.
\end{prop}

\begin{proof}
One checks this as in the proof of Lemma \ref{tilde one implies tilde n}.  The details are left to the reader.
\end{proof}

\begin{defn}
The category $\cat{PAlg}_0$ (respectively $\cat{SPAlg}_0$) is defined to be the quotient category $\cat{PAlg}/\sim_0$ (respectively $\cat{SPAlg}/\sim_0$).

The category $\cat{PAlg}_1$ (respectively $\cat{SPAlg}_1$) is defined to be the quotient category $\cat{PAlg}/\sim_1$ (respectively $\cat{SPAlg}/\sim_1$).
\end{defn}

Given a morphism $\alpha$ of $\cat{PAlg}$, we will sometimes denote the corresponding morphism of $\cat{PAlg}_0$ by $[\alpha]_0$ and of $\cat{PAlg}_1$ by $[\alpha]_1$.

The following lemma will be useful in the sequel.  It says that surjective morphisms in $\cat{PAlg}_1$ and $\cat{PAlg}_0$ are epimorphisms.

\begin{lemma} \label{CancellationProp}
Let $\alpha:A\to B$ be a surjective algebra homomorphism and $\beta,\beta':B\to C$ algebra homomorphisms such that $\beta\alpha\sim_1 \beta'\alpha$.  Then 
$\beta \sim_1 \beta'$.
\end{lemma}
\begin{proof}
Since $\beta\alpha\sim_1 \beta'\alpha$ we have
\[
 	(\beta-\beta')(B) = (\beta-\beta')\alpha(A) = (\beta\alpha-\beta'\alpha)(A)\subseteq J(C),
 \]
 \[
 (\beta-\beta')(J(B)) = (\beta-\beta')(\alpha(J(A)) = (\beta\alpha-\beta'\alpha)(J(A))\subseteq J^2(C).
 \]
\end{proof}

Working in the quotient category $\cat{PAlg}_1$ rather than $\cat{PAlg}$, much of the important information is preserved.  We check, for instance, that the canonical functor $\cat{PAlg}\to \cat{PAlg}_1$ reflects isomorphisms (that is, if a morphism $\alpha$ in $\cat{PAlg}$ is such that $[\alpha]_1$ is an isomorphism, then $\alpha$ itself is an isomorphism).

\begin{prop}
Given $\alpha:A\to A\in \cat{PAlg}$, if $\alpha\sim_1\tn{id}_A$, then $\alpha$ is an isomorphism.  Hence the projection functor $\cat{PAlg}\to \cat{PAlg}_1$ reflects isomorphisms.
\end{prop}

\begin{proof}
We check that $\alpha$ is injective.  If $\alpha(x)=0$, then
\[(\tn{id}_A - \alpha)(x) = x,\]
hence $x\in J(A)$.  Repeating, we see that $x\in J^2(A)$.  Continuing in this way and using Lemma \ref{tilde one implies tilde n} it follows that $x\in \bigcap_{n}J^n(A) = 0$.  For each $n\in \N$, note that the induced map $\overline{\alpha}:A/J^n\to A/J^n$ is such that $\overline{\alpha}\sim_1\tn{id}_{A/J^n}$.  But $A/J^n$ is finite dimensional and hence $\overline{\alpha}$ is an isomorphism by the argument above.  Now by Proposition \ref{prop complete wrt J}, $\alpha$ is an isomorphism.  The second claim follows formally.
\end{proof}

\section{Functors}\label{section functors}

\subsection{Functors from \cat{IQuiv}}

\subsubsection*{The functor $V(-)$}

Given a finite quiver $Q = (Q_0,Q_1)$, we define a Vquiver $VQ$ in the obvious way: the vertex set of $VQ$ is $Q_0\cup\{*\}$.  Given vertices $e,f$, the vector space $VQ_{e,f}$ is the space with basis the arrows $e\to f$ in $Q$ (the spaces $VQ_{e,*}, VQ_{*,e}$ are of course $0$).  Given a map $\iota:Q\to R$ in $\cat{IQuiv}$, which we consider as an inclusion for simplicity, we define the map $V(\iota):VR\to VQ$ as follows: send the vertex $e\in VR_0$ to itself if it is contained in $Q$, or to $*$ otherwise.  Given an arrow $a:e\to f$ of $R$ in $VR$, send $a$ to itself if it is an arrow of $VR$, or to $0_{VQ_{e,f}}$ otherwise.  This defines a map on the basis of the arrow spaces, and hence a (surjective) map of Vquivers.  We obtain in this way the contravariant functor $V(-):\cat{IQuiv}\to \cat{SVQuiv}$.

\medskip

\subsubsection*{The functor Completed path algebra}

This is a classical construction.  We observe that it is functorial.  Our main interest is in functors from $\cat{VQuiv}$, and so we leave checks of technical details to the reader.

%[ref for completed paths algebra] 
Given the quiver $Q = (Q_0,Q_1)$, let $CPA(Q)$ (\ospeech completed path algebra of $Q$\cspeech) be the pseudocompact $k$-vector space having basis the discrete set of all finite paths in $Q$.  The product of the paths $a_m\hdots a_1$ and $b_n\hdots b_1$ is given by the concatenation $a_m\hdots a_1 b_n\hdots b_1$ when $\mathrm{target}(b_n) = \mathrm{source}(a_1)$ and 0 otherwise.  This defines the structure of a pseudocompact algebra on $CPA(Q)$.  We have that $CPA(Q)/J^2(CPA(Q))$ has dimension $|Q_0| + |Q_1|<\infty$ and $CPA(Q)/J(CPA(Q))\iso\prod_{Q_0}k$, hence $CPA(Q)$ is indeed an object of $\cat{SPAlg}$.

Let $\iota:Q\to R$ be a morphism in $\cat{IQuiv}$.  Define the map $CPA(\iota): CPA(R)\to CPA(Q)$ on the path basis by sending a path of $CPA(R)$ to itself when it is completely contained in $Q$, and to 0 otherwise.  This yields a surjective map of algebras.  In this way we obtain a contravariant functor $CPA(-):\cat{IQuiv}\to \cat{SPAlg}$.

\medskip

Observe that when the quivers are acyclic, the algebra $CPA(Q) = kQ$ is just the usual path algebra.  But even with acyclic quivers, the obvious inclusion map $kQ\to kR$ induced by $\iota:Q\to R$ is \emph{not} in general an algebra homomorphism, because it does not send $1_{kQ}$ to $1_{kR}$ if $\iota$ is not surjective on vertices.  The inclusion $kQ\to kR$ is instead a morphism of coalgebras.  We suggest that this helps to explain why the theory of path coalgebras for infinite quivers (see for instance \cite{SimsonCoalg}) has made more progress than the theory of path algebras for infinite quivers.  To work with algebras, one should treat \ospeech path algebra\cspeech as a contravariant functor and work with pseudocompact algebras.

\subsection{Functors from $\cat{VQuiv}$}\label{subsection completed path algebra}

Given a finite Vquiver $VQ = (VQ_0^*, VQ_{e,f})$, define the following objects:
\begin{itemize}
\item The semisimple pointed algebra 
\[\Sigma_{VQ} = \prod_{e\in VQ_0}k.\]
\item The vector space $VQ_{1} = \bigoplus_{e,f\in VQ_0}VQ_{e,f}$ treated as a $\Sigma_{VQ}$-bimodule with multiplication from $(\lambda_g)_{g\in VQ_0}\in \Sigma_{VQ}$ on an element $x$ of $VQ_{e,f}$ defined as
\[x\cdot(\lambda_g) = x\lambda_e\quad,\quad (\lambda_g)\cdot x = \lambda_fx.\]
\end{itemize}
We send $VQ$ to $k\db{VQ} = T\db{\Sigma_{VQ},VQ_{1}}$. By Lemma \ref{Lemma radical of CPA what we expect} we have that $k\db{VQ} \in \cat{PAlg}$.

\medskip

Given a map of Vquivers $\rho:VR\to VQ$, define a map $\alpha_0':\Sigma_{VR}\to \Sigma_{VQ}$ on the obvious basis of $VR_0$ by sending $e$ to $\rho(e)$ if $\rho(e)\neq *$ and to 0 otherwise.  Summing the linear maps $VR_{e,f} \to VQ_{\rho(e),\rho(f)}$ we obtain a linear transformation $\alpha_1':VR_{1}\to VQ_{1}$, which is easily checked to be a $\Sigma_{VQ}$-bimodule homomorphism via $\alpha_0'$.  Now, regarding $\Sigma_{VQ}$ and $VQ_{1}$ as subspaces of $T\db{\Sigma_{VQ},VQ_{1}}$, we obtain by composing with the inclusions a continuous algebra homomorphism $\alpha_0:\Sigma_{VR}\to T\db{\Sigma_{VQ},VQ_{1}}$ and a continuous $\Sigma_{VR}$-bimodule homomorphism $VR_{1}\to T\db{\Sigma_{VQ},VQ_{1}}$.  By the universal property of the completed tensor algebra, we gain a unique continuous algebra homomorphism $T\db{\Sigma_{VR},VR_{1}}\to T\db{\Sigma_{VQ},VQ_{1}}$. Note that this homomorphism is surjective when $\rho$ is surjective (this follows by construction, as $\alpha_0'$ and $\alpha_1'$ are surjective in this case). 

In this way we obtain a covariant functor $k\db{-}:\cat{VQuiv}\to \cat{PAlg}$ (whose restriction to $\cat{SVQuiv}$ is a functor to $\cat{SPAlg}$).

\medskip

Observe that $CPA(-)$ 
%restricted to $\cat{IQuiv}$ 
is the composition $k\db{-}\circ V(-)$.

\medskip

One obtains the functor $\widetilde{k}\db{-}:\cat{VQuiv}\to \cat{PAlg}_1$ by composing $k\db{-}$ with the canonical projection $\Pi_1:\cat{PAlg}\to \cat{PAlg}_1$.

\subsection{Functors from $\cat{PAlg}$ and $\cat{PAlg}_1$}

We make the Gabriel quiver construction (considered as a map to Vquivers) functorial.

\medskip

Let $A$ be a pseudocompact pointed $k$-algebra with $A/J^2$ of finite dimension.  Denote by $\mathcal{S}_A$ the non-empty set of splittings of the canonical projection $A\onto A/J$ as in Proposition \ref{WedderburnMalcev}.  The algebra $A/J\iso \prod_{1\leqslant i\leqslant n}k$ has a unique complete set of primitive orthogonal idempotents $P$.   The image $s(P)$ under any splitting yields a complete set of primitive orthogonal idempotents of $A$.  Denoting by $\mathcal{P}_A$ the set of all complete sets of primitive orthogonal idempotents of $A$, we obtain a map $\Omega:\mathcal{S}_A\to \mathcal{P}_A$.

\begin{lemma}
The map $\Omega$ is bijective.
\end{lemma}

\begin{proof}
The map is injective because given $s,t\in \mathcal{S}_A$, if $s(P) = t(P)$ then for each $e\in P$, $s(e), t(e)$ are either equal or orthogonal.  By Proposition \ref{WedderburnMalcev}, there is $w\in J$ such that $s(e) = {}^{1+w}t(e) = t(e) + j$ for some $j\in J$.  But $t(e)(t(e)+j) = t(e)+j'$ (some $j'\in J$), which is not 0.  So $s(e), t(e)$ are not orthogonal, and hence are equal.

The map is surjective because given a set $\{f_1,\hdots,f_n\}\in \mathcal{P}_A$, the image $\{f_1+J,\hdots,f_n+J\}$ must be $P$.  The inverse map $f_i+J\mapsto f_i$ yields a linear transformation $A/J\to A$ that is easily confirmed to be a splitting. 
\end{proof}

With notation as in Proposition \ref{WedderburnMalcev}, denote by $\mathcal{G}(A)$ the following subgroup of $\tn{Aut}(A)$:
\[\mathcal{G}(A) = \left\{{}^{1+w}(-)\,|\,w\in J(A)\right\}.\]
When $A$ is understood, we denote this group simply by $\mathcal{G}$.  Given an element $a\in A$, denote by ${}^{\mathcal{G}}a$ its orbit under $\mathcal{G}$.

\begin{defn}\label{GQ on objects}
Let $A$ be an object of $\cat{PAlg}$, $s\in \mathcal S_A$ and $\Omega(s)\in \mathcal{P}_A$ the corresponding set of primitive orthogonal idempotents in $A$. Define the Vquiver $\tn{GQ}(A)$ of $A$ as follows:
\begin{align*}
	\tn{GQ}(A)_0^* & :=\{*\}\cup \{{}^{\mathcal G}e\,|\,e\in \Omega(s)\},\\
    \tn{GQ}(A)_{{}^{\mathcal G}e,{}^{\mathcal G}f} & := f\frac{J(A)}{J^2(A)}e.
\end{align*}
%[[Observe that we use $\frac{a}{b}$ to mean quotient]]
\end{defn}

We must check that this Vquiver is well-defined.  That is, that it does not depend on the choice of $s\in \mathcal{S}_A$:

\begin{lemma}
The Vquiver $\tn{GQ}(A)$ is well-defined.
\end{lemma}

\begin{proof}
Proposition \ref{WedderburnMalcev} guarantees that the objects do not depend on $s$.

We must check that given $w,w'\in J$ and $e,f\in \Omega(s)$ we have
\[{}^{1+w}f\frac{J(A)}{J^2(A)}{}^{1+w'}e = f\frac{J(A)}{J^2(A)}e.\]
We prove that ${}^{1+w}f\frac{J(A)}{J^2(A)}e = f\frac{J(A)}{J^2(A)}e$, the general case following in the same spirit.  Observe that $(1+w)^{-1} = 1+z$ for some $z\in J$ and hence for any $j\in J$ we have $({}^{1+w}f)je =fje + y$ for some $y\in J^2$.  Now
\[({}^{1+w}f)(j+J^2)e = ({}^{1+w}f)je+J^2 = fje+J^2 = f(j + J^2)e.\]
The required equality follows.
\end{proof}

This defines $\tn{GQ}(A)$ on objects.  The action on morphisms is the obvious one:

\begin{defn}\label{GQ on morphisms}
Let $\alpha:A\to B$ be a morphism in $\cat{PAlg}$.  
%Write $\mathcal{G} = \mathcal{G}(A)$ and $\mathcal{H} = \mathcal{G}(B)$. 
Define the map $\tn{GQ}(\alpha):\tn{GQ}(A)\to \tn{GQ}(B)$ as follows:
\begin{itemize}
\item On vertices, $\tn{GQ}(\alpha)\left({}^{\mathcal{G}(A)}e\right) = \begin{cases}
{}^{\mathcal{G}(B)}\alpha(e) & \alpha(e)\neq 0, \\
* & \alpha(e)=0.
\end{cases} 
$
\item On arrow spaces, the map $\tn{GQ}(\alpha)\,:\,  f\dfrac{J(A)}{J^2(A)}e\to \alpha(f)\dfrac{J(B)}{J^2(B)}\alpha(e)$ sends $f(j+J^2(A))e$ to 
$\alpha(f)(\alpha(j)+J^2(B))\alpha(e)$.
\end{itemize}
\end{defn}

\begin{lemma}
Given $\alpha$ a morphism in $\cat{PAlg}$, $\tn{GQ}(\alpha)$ as defined above is a well-defined map of Vquivers. Moreover if $\alpha$ is surjective then $\tn{GQ}(\alpha)$ is surjective as well.
\end{lemma}

\begin{proof}
Let $s$ be an element of $\mathcal{S}_A$ and $\Omega(s)$ the corresponding complete set of primitive orthogonal idempotents of $A$.  It is well-known and easily checked that $\alpha$ maps a subset of $\Omega(s)$ bijectively onto some element of $\mathcal{P}_B$ and the rest to 0.  It follows that $\tn{GQ}(\alpha)$ makes sense on vertices.  It is immediate that $\tn{GQ}(\alpha)$ makes sense on arrow spaces. Supposing that $\alpha$ is surjective we check that $\tn{GQ}(\alpha)$ is surjective as well.  Fix $w\in J(B)$.  By Corollary \ref{radical surjective}, there is $j\in J(A)$ with $\alpha(j)=w$.  The element $\alpha(f)(w+J^2(B))\alpha(e) = \tn{GQ}(\alpha)(f(j+J^2(A))e)$ and hence $\tn{GQ}(\alpha)$ is surjective, as required. 
\end{proof}

The operation $\tn{GQ}(-)$ defined on objects in Definition \ref{GQ on objects} and on morphisms in Definition \ref{GQ on morphisms} defines a covariant functor $\tn{GQ}(-):\cat{PAlg}\to \cat{VQuiv}$.

We will factorize $\tn{GQ}(-)$ through the category $\cat{PAlg}_1$.  To do so, we need the following simple fact about primitive idempotents.

\begin{lemma}\label{idempotent orbit characterization lemma}
Let $e,f$ be primitive idempotents of $A\in \cat{PAlg}$.  Then ${}^{1+w}e=f$ for some $w\in J(A)$ if, and only if, $e-f\in J(A)$.
\end{lemma}

\begin{proof}
The forward implication is obvious.  Suppose that $e-f=j$ for some $j\in J$.  Let $\Sigma, \Sigma'$ be semisimple subalgebras of $A$ isomorphic to $A/J$ and containing $e,f$ respectively.  Then by Proposition \ref{WedderburnMalcev} there is some $w\in J$ with ${}^{1+w}\Sigma = \Sigma'$.  Hence either ${}^{1+w}e$ is equal to $f$ or their product is 0.  But
\[{}^{1+w}ef = {}^{1+w}e(e+j) = e+j'\qquad (\hbox{some }j'\in J)\]
and so ${}^{1+w}e = f$, as required.
\end{proof}

\begin{prop} \label{unique functor from SPAlg1}
There exists a unique functor $\widetilde{\tn{GQ}}(-):\cat{PAlg}_1\to \cat{VQuiv}$ such that $\tn{GQ}(-) = \widetilde{\tn{GQ}}(-)\circ \Pi_1$, where $\Pi_1:\cat{PAlg}\to \cat{PAlg}_1$ is the canonical projection.
\end{prop}

\begin{proof}
We define $\widetilde{\tn{GQ}}(-)$ as follows.  On objects, $\widetilde{\tn{GQ}}(A) = \tn{GQ}(A)$.  Given a morphism $[\alpha]:A\to B$ in $\cat{PAlg}_1$ with representative $\alpha$ in $\cat{PAlg}$, define $\widetilde{\tn{GQ}}([\alpha]) = \tn{GQ}(\alpha)$.  We must check that this definition does not depend on the representative of $[\alpha]$, so consider two representatives $\alpha, \alpha'\in [\alpha]$.  To confirm that $\widetilde{\tn{GQ}}([\alpha])$ is well-defined on vertices, we need to check that ${}^{\mathcal{G}(B)}\alpha(e) = {}^{\mathcal{G}(B)}\alpha'(e)$ for any primitive idempotent $e$ of $A$.  But $\alpha(e) - \alpha'(e) = (\alpha-\alpha')(e)\in J(B)$ since $\alpha\sim_1\alpha'$ and hence by Lemma \ref{idempotent orbit characterization lemma}, ${}^{\mathcal{G}(B)}\alpha(e) = {}^{\mathcal{G}(B)}\alpha'(e)$.  We still need to check that the maps $\tn{GQ}(\alpha)$ and $\tn{GQ}(\alpha')$ agree on arrow spaces $\tn{GQ}(A)_{e,f}$, but using that $(\alpha-\alpha')(A)\subseteq J(B)$ and $(\alpha-\alpha')(J(A))\subseteq J^2(B)$, for $j\in J$ we have
	\begin{align*}
		\tn{GQ}(\alpha')(f(j+J^2(A))e) & = \alpha'(f)(\alpha'(j)+J^2(B))\alpha'(e)\\
								&= \alpha(f)(\alpha(j)+J^2(B))\alpha(e)\\
								&= \tn{GQ}(\alpha)(f(j+J^2(A))e).
	\end{align*}
\end{proof}

\subsection{Examples}

To help our intuition with the definitions of Sections \ref{section categories} and \ref{section functors}, we briefly present some simple examples.
\begin{itemize}
\item The algebra $k\db{VQ}$ is semisimple if and only if $VQ$ has no edges.  If $B$ is semisimple and $A$ is arbitrary, then 
\[\tn{Hom}_{\cat{PAlg}}(A,B) = \tn{Hom}_{\cat{PAlg}_1}(A,B).\]
This equality is rare, though it can occur with non-semisimple algebras $B$.  For example, it is the case when $B = k\db{x}/x^2$.

\item Let $VQ$ be the Vquiver
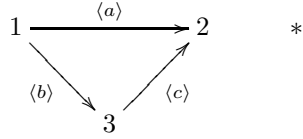
\[\xymatrix{1 \ar[dr]_{\langle b\rangle}\ar[rr]^{\langle a\rangle} && 2 & {*} \\ 
& 3\ar[ur]_{\langle c\rangle} &&}\]
Then $k\db{VQ}$ is $7$-dimensional with basis $e_1, e_2, e_3, a,b,c, cb$ (where $e_i$ is the idempotent at vertex $i$).  The map $\alpha$ defined on this basis to be the identity on the $e_i$ and on $b,c,cb$, but sending $a$ to $a+cb$ is an algebra automorphism.  Then $\alpha\sim_1\tn{id}$, because
\[(\alpha - \tn{id})(a) = a+cb-a = cb\in J^2.\]

\item We classify those completed path algebras $k\db{VQ}$ for which the set of representatives of the class $[\tn{id}]_1$ is exactly the group $\mathcal{G}(k\db{VQ})$.  They correspond to those quivers with the property that there does not exist a pair of vertices $x,y$ having a non-zero path ($\rho$) of length $1$ from $x$ to $y$ and also a non-zero path ($\theta$) of length greater than $1$ from $x$ to $y$.  One direction is clear.  For the other, simply note that the morphism defined to be the identity on idempotents and every arrow in a basis except $\rho$, but sending $\rho$ to $\rho + \theta$, defines an algebra homomorphism in $[\tn{id}]_1$ that is not in $\mathcal{G}(k\db{VQ})$.
\end{itemize}

\section{$\widetilde{k}\db{-}$ is left adjoint to $\widetilde{\tn{GQ}}(-)$}\label{section main adjunction}

We prove in this section our first main theorem: the functor $\widetilde{k}\db{-}$ is left adjoint to $\widetilde{\tn{GQ}}(-)$.  We do so by constructing a natural bijection
\[
	\Psi_{VQ,A}:\tn{Hom}_{\cat{VQuiv}}(VQ,\widetilde{\tn{GQ}}(A))\to \tn{Hom}_{\cat{PAlg}_1}(\widetilde{k}\db{VQ},A)
\]
for $VQ\in \cat{VQuiv}$ and $A\in \cat{PAlg}_1$.

\medskip

Fix (by Proposition \ref{WedderburnMalcev}) a splitting $s$ of $A\to A/J$ and let $\Sigma = s(A/J)$.  Fix also (by Lemma \ref{Lemma J to J by J2 splits}) a splitting $t$ of the projection $J\to J/J^2$ as $\Sigma$-bimodules.  The orbit under $\mathcal{G} = \mathcal{G}(A)$ of a primitive idempotent $f$ of $A$ intersects $\Sigma$ in exactly one point, which we denote by $f_{{}_{\Sigma}}$.  

Fix a morphism $\rho:VQ\to \widetilde{\tn{GQ}}(A)$.  The map defined on the basis $VQ_0$ by
\begin{align*}
	                  VQ_0 & \rightarrow A\\
			e\,\,\,\,      & \mapsto \rho(e)_{{}_{\Sigma}}
\end{align*}
yields a homomorphism of algebras $\varphi_0:\Sigma_{VQ}\to A$ (which depends on $s$).  For each pair $e,f\in VQ_0$, $\rho$ restricts to a map $VQ_{e,f}\to \rho(f)_{{}_{\Sigma}}\frac{J(A)}{J^2(A)}\rho(e)_{{}_{\Sigma}}$.  Composing with $t$ we obtain
\[VQ_{e,f}\xrightarrow{\rho} \rho(f)_{{}_{\Sigma}}\frac{J(A)}{J^2(A)}\rho(e)_{{}_{\Sigma}} \xrightarrow{t} A.\]
Summing these maps as $e,f$ vary, we obtain a $\Sigma_{VQ}$-bimodule homomorphism
\[\varphi_1 : VQ_1 \to A\]
(which depends on $t$ and $s$).  By the universal property of the completed tensor algebra \ref{Lemma UP tensor algebra}, the maps $\varphi_0,\varphi_1$ correspond to a unique homomorphism of algebras $\varphi_{s,t}:k\db{VQ}\to A$ (which is easily checked to be surjective if $\rho$ is surjective).

\begin{lemma}
The equivalence class of $\varphi_{s,t}$ in $\cat{PAlg}_1$ is independent of $s$ and $t$.
\end{lemma}

\begin{proof}
Let $s':A/J\to A$ (corresponding to the subalgebra $\Sigma'$) and $t':J/J^2\to J$ be other splittings.  In order to check that $(\varphi_{s,t} - \varphi_{s',t'})(k\db{VQ})\subseteq J(A)$, it suffices to check on primitive idempotents.  But $(\varphi_{s,t} - \varphi_{s',t'})(e) = \rho(e)_{{}_{\Sigma}} - \rho(e)_{{}_{\Sigma'}}$, which is an element of $J(A)$ by Lemma \ref{idempotent orbit characterization lemma}.

Since $J(k\db{VQ}) = VQ_1\oplus J^2(k\db{VQ})$, it remains to confirm that $(\varphi_{s,t} - \varphi_{s',t'})(VQ_1)\subseteq J^2(A)$.  Both $t$ and $t'$ are splittings of $J(A)\to J(A)/J^2(A)$ and hence $(t-t')(x)\in J^2(A)$ for any $x\in J(A)/J^2(A)$.  For any $v\in VQ_1$ it follows that
\[(\varphi_{s,t} - \varphi_{s',t'})(v) = (t-t')\rho(v) \in J^2(A).\]
Thus $[\varphi_{s,t}]_1 = [\varphi_{s',t'}]_1$, as required.
\end{proof}

It follows that the correspondence $\rho\mapsto [\varphi_{s,t}]_1$ yields a well-defined function
\[
	\Psi_{VQ,A}:\tn{Hom}_{\cat{VQuiv}}(VQ,\widetilde{\tn{GQ}}(A))\to \tn{Hom}_{\cat{PAlg}_1}(\widetilde{k}\db{VQ},A).
\]

\begin{lemma} \label{Bijectivity of Phi.}
The map $\Psi_{VQ,A}$ is bijective.
\end{lemma}

\begin{proof}
We first check surjectivity.  Fix an algebra homomorphism $\alpha:k\db{VQ}\to A$.  Restricting $\alpha$ to $\Sigma_{VQ}$ and to $VQ_1$ yields a map of Vquivers $\rho:VQ\to \widetilde{\tn{GQ}}(A)$. %(the images of $\Sigma_{VQ}$ and of $VQ_1$ determining splittings $s:A/J\to A$ and $t: J/J^2\to J$, respectively).  
Now $\Psi_{VQ,A}(\rho) = [\alpha]_1$ and so $\Psi_{VQ,A}$ is onto.

It remains to check injectivity.  Suppose that $\rho, \rho':VQ\to \widetilde{\tn{GQ}}(A)$ are distinct Vquiver maps. Denote by $\varphi_{s,t}, \varphi'_{s,t}$ the homomorphisms constructed from $\rho, \rho'$ respectively.  Then either 
\begin{itemize}
\item $\rho(e)\neq\rho'(e)$ for some $e\in VQ_0$.  But then $\rho(e)_{{}_{\Sigma}}$ and $\rho'(e)_{{}_{\Sigma}}$ are in different $\mathcal{G}(A)$-orbits, so that $\varphi_{s,t}\not\sim_1\varphi'_{s,t}$.

\item or the maps agree on vertices but there is $v\in VQ_1$ such that $\rho(v)\neq\rho'(v)$.  Now
\[(\varphi_{s,t} - \varphi'_{s,t})(v) = t(\rho-\rho')(v)\not\in J^2\]
because $(\rho-\rho')(v)\neq 0$.  Again, we have $\varphi_{s,t}\not\sim_1\varphi'_{s,t}$.
\end{itemize}
\end{proof}

We must still check naturality in both variables.  Naturality in the first variable is a routine calculation.  We prove naturality in the second:

\begin{lemma} \label{naturality at A in first adjunction}
Fix a  Vquiver $VQ$.  The map
\[\Psi_{VQ,A}:\tn{Hom}_{\cat{VQuiv}}(VQ,\widetilde{\tn{GQ}}(A))\to \tn{Hom}_{\cat{PAlg}_1}(\widetilde{k}\db{VQ},A)
\]
is the component at $A$ of a natural transformation
\[\Psi_{VQ,-}:\tn{Hom}_{\cat{VQuiv}}(VQ,\widetilde{\tn{GQ}}(-))\to \tn{Hom}_{\cat{PAlg}_1}(\widetilde{k}\db{VQ},-).
\]
\end{lemma}

\begin{proof}
Fix $[\alpha]_1:A \to B$ in $\cat{PAlg}_1$.  We need to confirm that the diagram
\[\xymatrix{
\tn{Hom}_{\cat{VQuiv}}(VQ,\widetilde{\tn{GQ}}(A))\ar[rr]^{\Psi_{VQ,A}}\ar[d]_{\widetilde{\tn{GQ}}([\alpha]_1)\circ-} & & \tn{Hom}_{\cat{PAlg}_1}(\widetilde{k}\db{VQ},A)\ar[d]^{[\alpha]_1\circ-} \\
\tn{Hom}_{\cat{VQuiv}}(VQ,\widetilde{\tn{GQ}}(B))\ar[rr]_{\Psi_{VQ,B}} & & \tn{Hom}_{\cat{PAlg}_1}(\widetilde{k}\db{VQ},B)
}\]
commutes.  Fix $\rho:VQ\to \widetilde{\tn{GQ}}(A)$ and $\alpha\in [\alpha]_1$.  Given two splittings $s:A/J(A)\to A$, $s':B/J(B)\to B$  we have that $\tn{Im}(\alpha s-s'\alpha)\subseteq J(B)$ (where the latter $\alpha$ abusively denotes the map $A/J(A)\to B/J(B)$ induced by $\alpha$).

For any $e\in VQ_0$ with $\rho(e)\neq *$ (the case $\rho(e)=*$ being easier), there is $w\in J(B)$ such that 
\begin{align*}
\Psi_{VQ,B}({\tn{GQ}}(\alpha)\rho)(e) -\alpha\Psi_{VQ,A}(\rho)(e)& = {\tn{GQ}}(\alpha)\rho(e)_{{s'\alpha(A/J)}} -\alpha(\rho(e)_{{s(A/J)}})\\
& = {}^{\mathcal{G}(B)}\alpha\rho(e)_{{s'\alpha(A/J)}} - \alpha(\rho(e)_{{s(A/J)}})\\
& = w+\alpha(\rho(e)_{{s(A/J)}}) - \alpha(\rho(e)_{{s(A/J)}})\\
& = w.
\end{align*}
Also for two splittings $t:J(A)/J^2(A)\to J(A)$ and $t':J(B)/J^2(B)\to J(B)$ we have that $\textrm{Im}(\alpha t-t'\alpha)\subseteq J^2(B)$ (where again the latter $\alpha$ abusively denotes the map $J(A)/J^2(A)\to J(B)/J^2(B)$ induced by $\alpha$).  Now given $v\in VQ_1$ we have
\begin{align*}
\Psi_{VQ,B}({\tn{GQ}}(\alpha)\rho)(v) - \alpha\Psi_{VQ,A}(\rho)(v)& = t'{\tn{GQ}}(\alpha)\rho(v) - \alpha t\rho(v)\\
& = t'\alpha\rho(v) - \alpha t\rho(v) \in J^2(B).
\end{align*}
Therefore we have that $(\Psi_{VQ,B}\circ{\tn{GQ}}(\alpha))(\rho) \sim_1 (\alpha\circ\Psi_{VQ,A})(\rho)$ for any $\rho$ and any $\alpha\in [\alpha]_1$.
\end{proof}

\noindent Summarizing,

\begin{theorem}\label{Theorem Adjunction Level 1}
The functor $\widetilde{k}\db{-}:\cat{VQuiv}\to \cat{PAlg}_1$ is left adjoint to the functor $\widetilde{\tn{GQ}}(-):\cat{PAlg}_1\to \cat{VQuiv}$.
\end{theorem}

\begin{proof}
The natural isomorphism required is $\Psi$ as defined above.  The results of this section demonstrate the theorem.
\end{proof}

\begin{remark} Restricting and corestricting the functors above, we obtain an adjuntion between the categories $\cat{SVQuiv}$ and $\cat{SPAlg}_1$.
\end{remark}

\begin{remark} A pseudocompact algebra $A$ is said to be \emph{hereditary} if every closed submodule of a projective $A$-module is projective.  It is known (cf. \cite[Theorem 1]{Chin}) that $A$ is hereditary if, and only if, $A$ is isomorphic to $\widetilde{k}\db{\widetilde{\tn{GQ}}(A)}$.  It follows that if we restrict $\cat{PAlg}_1$ to the full subcategory of hereditary algebras, the Adjunction \ref{Theorem Adjunction Level 1} yields an adjoint equivalence of categories.
\end{remark}

\begin{remark}
The Adjunction \ref{Theorem Adjunction Level 1} restricts to an adjunction between the full subcategory of $\cat{VQuiv}$ consisting of acyclic Vquivers and the full subcategory of $\cat{PAlg}_1$ consisting of the triangular finite dimensional algebras (that is, those finite dimensional algebras whose Gabriel quiver is acyclic). 
\end{remark}

\begin{remark} \label{remGQfaithful}
For any algebra $A\in \cat{PAlg}_1$, the morphism $\widetilde{k}\db{\widetilde{\tn{GQ}}(A)}\to A$ is an epimorphism. It follows by abstract nonsense \cite[IV.3, Theorem 1]{MacLane} that the functor $\widetilde{\tn{GQ}}(-)$ is faithful.
\end{remark}

Recall 
%[cite] 
that a closed ideal $I\in k\db{VQ}$ is called a \emph{relation ideal} if $I\subseteq J^2$. A relation ideal $I\in k\db{VQ}$ is \emph{admissible} if there exists a positive integer $n$ such that $J^n\subseteq I$.

\begin{prop}
Every algebra $A$ in $\cat{PAlg}$ is isomorphic to the quotient of a completed path algebra $k\db{VQ}/I$ with $I$ a relation ideal.  The relation ideal $I$ is admissible if, and only if, $A$ is finite dimensional.
\end{prop}

\begin{proof}
Write $J^n = J^n(\widetilde{k}\db{\widetilde{GQ}(A)})$. Any representative of the counit $\varepsilon_A:\widetilde{k}\db{\widetilde{GQ}(A)}\to A$ of the Adjunction \ref{Theorem Adjunction Level 1} yields such a quotient.  Let $I$ be the kernel of such a representative. Then, by construction of $\Psi_{\widetilde{k}\db{\widetilde{GQ}(A)},A}$, $I$ is a closed ideal inside $J^2$, hence it is a relation ideal.

If $A$ is not finite dimensional, then the kernel of $\varepsilon_A$ cannot contain any $J^n$, since the algebra $\widetilde{k}\db{\widetilde{GQ}(A)}/J^n$ is finite dimensional for every $n$.  On the other hand, if $A$ is finite dimensional, then the radical $J(A)$ is nilpotent and so $J(A)^n = 0$ for some $n$.  Since $J^n$ maps onto $J(A)^n$ under $\varepsilon_A$, the result follows. 
\end{proof}

\begin{lemma}\label{unit is isomorphism}
Let $VQ$ be a finite Vquiver and $I$ a relation ideal of $VQ$. If $\alpha\in \tn{Hom}_{\cat{SPAlg}_1}(\widetilde{k}\db{VQ},\widetilde{k}\db{VQ}/I)$ then 
$\Psi^{-1}_{VQ,\widetilde{k}\db{VQ}/I}(\alpha)$ is an isomorphism. 
In particular, the unit of Adjunction \ref{Theorem Adjunction Level 1} is a natural isomorphism.
\end{lemma}

\begin{proof}
The Vquivers $VQ$ and $\widetilde{\tn{GQ}}(\widetilde{k}\db{VQ}/I)$ are isomorphic by construction.  A surjective map of isomorphic finite Vquivers is an isomorphism.  
\end{proof}

\begin{remark}
As in Remark \ref{remGQfaithful}, it follows from Lemma \ref{unit is isomorphism} and abstract nonsense that $\widetilde{k}\db{-}$ is fully faithful.
\end{remark}

\section{Related adjunctions}\label{section related adjunctions}

We mention two related adjunctions, interesting in their own right.

\subsection{Sets and Semisimple algebras}

Denote by $\cat{PSet}$ the category of pointed finite sets defined as follows.  Denote the objects by $Q_0^* = Q_0\cup\{*\}$.  A morphism $Q_0^*\to R_0^*$ is a pointed map that restricts to a bijection from the elements of $Q_0$ not mapping to $*$ onto $R_0$.  One can regard $\cat{PSet}$ as the full subcategory of $\cat{VQuiv}$ whose objects are those Vquivers in which every arrow space is 0.  One could alternatively describe $\cat{PSet}$ as the category of finite dimensional vector spaces over the field with one element, with morphisms surjective linear maps \cite[Definition 2.1]{szczesny}.

\medskip

Denote by $k_0\db{-}$ the restriction of $k\db{-}$ to the subcategory $\cat{PSet}$.  Explicitly, 
\[k_0\db{Q_0^*} = \prod_{e\in Q_0}k\]
and given $\rho:VQ_0^*\to VR_0^*$, we send $e\in VQ_0$ to $\rho(e)$ if $\rho(e)\neq *$ and to 0 otherwise.  Denote by $\widetilde{k}_0\db{-}$ the composition 
\[\cat{PSet}\xrightarrow{k_0\db{-}} \cat{PAlg} \xrightarrow{\Pi_0}\cat{PAlg}_0\]
where $\Pi_0$ is the canonical projection.

\medskip

Given $A\in \cat{PAlg}$, define the pointed set $\tn{GQ}_0(A)$ to be the set of $\mathcal{G}(A)$-orbits of a complete set of primitive orthogonal idempotents of $A$, together with the point.  A surjective algebra homomorphism $\alpha:A\to B$ yields a well-defined pointed map $\tn{GQ}_0(\alpha):\tn{GQ}_0(A)\to \tn{GQ}_0(B)$ as in the definition of $\tn{GQ}(-)$.  We obtain in this way the functor $\tn{GQ}_0(-):\cat{PAlg}\to \cat{PSet}$.  This functor factorizes uniquely as $\tn{GQ}_0(-) = \widetilde{\tn{GQ}}_0(-)\circ\Pi_0$.  One now easily checks the following:

\begin{prop}
The functor $\widetilde{\tn{GQ}}_0(-):\cat{PAlg}_0\to \cat{PSet}$ is left adjoint to the functor $\widetilde{k}_0\db{-}:\cat{PSet}\to \cat{PAlg}_0$.
\end{prop}

\subsection{A right adjoint to $\widetilde{\tn{GQ}}(-)$}

The functor $\widetilde{\tn{GQ}}(-)$ is also a left adjoint.  We construct its right adjoint.

Given a finite Vquiver $VQ$, let $k_2\db{VQ} = k\db{VQ}/J^2(k\db{VQ})$.  A map of Vquivers $\rho:VQ\to VR$ yields a  map of algebras  $k\db{\rho}:k\db{VQ}\to k\db{VR}$, which sends $J^2(k\db{VQ})$ to $J^2(k\db{VR})$.  One thus obtains from $k\db{\rho}$ the algebra homomorphism
\[k_2\db{\rho}:k_2\db{VQ}\to k_2\db{VR}\]
in the obvious way.  These definitions yield the covariant functor $k_2\db{-}:\cat{VQuiv}\to \cat{PAlg}$.  By composing with $\Pi_1$ we obtain the functor
\[\widetilde{k}_2\db{-}:\cat{VQuiv}\to \cat{PAlg}_1.\]

\begin{prop}
The functor $\widetilde{k}_2\db{-}$ is right adjoint to the functor $\widetilde{\tn{GQ}}(-)$.
\end{prop}

\begin{proof}
We define the map
\[\Phi_{A,VQ}:\tn{Hom}_{\cat{VQuiv}}(\widetilde{\tn{GQ}}(A), VQ)\to \tn{Hom}_{\cat{PAlg}_1}(A, \widetilde{k}_2\db{VQ})\]
as follows.  Fix a Vquiver map $\rho:\widetilde{\tn{GQ}}(A)\to VQ$.  We construct an algebra homomorphism $\alpha: A\to \widetilde{k}_2\db{VQ}$.  Choose splittings $s:A/J\to A$ and $t:J/J^2\to J$ of the canonical projections as in Section \ref{section main adjunction}, so that
\[A = s(A/J)\oplus t(J/J^2)\oplus J^2.\]
Let $e$ be a primitive idempotent of $A$, $j\in J$ and $w\in J^2$.   Define 
\[\alpha(e,j,w) := \rho({}^{\mathcal{G}}e) + \rho(j+J^2).\]
One checks that if $s',t'$ are different splittings, the map $\beta$ obtained is such that $\alpha\sim_1\beta$, and hence $\alpha$ is well-defined in $\cat{PAlg}_1$. The careful reader may check that $\Phi_{A,VQ}$ is a natural bijection, completing the proof.
\end{proof}

\section{Adjoint equivalence with relation ideals}\label{section adjoint equivalence}

In this final section, we show that with some care, the adjunction of Section \ref{section main adjunction} can be shown to respect ideals in a sense we make precise.  We define a functor $\tn{GQ}_{\infty}(-)$ from $\cat{PAlg}_1$ to a category whose objects are pairs $(VQ,[I])$ with $VQ$ a Vquiver and $[I]$ a certain equivalence class of relation ideals of $k\db{VQ}$.  We show that it is an equivalence of categories.  Restricting to surjective maps in both categories, we exhibit the left adjoint to $\tn{GQ}_{\infty}$.

\subsection{Morphisms close to the identity}

Given a Vquiver $VQ$, denote by $[\id_{VQ}]_1$ the equivalence class in $\cat{PAlg}$ of $\tn{id}_{k\db{VQ}}$ under $\sim_1$. It is easily checked that $[\id_{VQ}]_1$ is a normal subgroup of $\tn{Aut}(k\db{VQ})$.  We give a convenient characterization of the $\sim_1$ equivalence classes when the source is a completed path algebra:

\begin{prop} \label{tilde is controlled on VQuiver side}
Let $VQ$ be an object of $\cat{VQuiv}$ and $\alpha,\beta:k\db{VQ} \to A$ two surjective morphisms in $\cat{PAlg}$. The following are equivalent:
\begin{enumerate}
\item $\alpha \sim_1 \beta$.
\item There is $\delta \in [\id_{VQ}]_1$ such that $\alpha=\beta\delta$.
\end{enumerate}
\end{prop}
\begin{proof}

It is obvious that the second condition implies the first.  Suppose that $\alpha\sim_1 \beta$.  Writing $k\db{VQ} = T\db{\Sigma, V}$ with notation as in Section \ref{subsection completed path algebra}, it follows that 
\[(\alpha-\beta)(\Sigma)\subseteq J(A),\quad (\alpha-\beta)(V)\subseteq J^2(A).\]
By Proposition \ref{WedderburnMalcev} (and arguments similar to those in Lemma \ref{idempotent orbit characterization lemma}) there is $w\in J(A)$ such that $\alpha(x)=(1+w)\beta(x)(1+w)^{-1}$ for any $x\in \Sigma$. Let $v$ be an element of $J(k\db{VQ})$ such that $\beta(v)=w$. 
Define the automorphism $\delta^{(1)}$ of $k\db{VQ}$ by $\delta^{(1)}(x)=(1+v)x(1+v)^{-1}$ for any $x\in k\db{VQ}$. Observe that  
$\alpha=\beta\delta^{(1)}$ on $\Sigma$ and for any $j\in V$ we have 
\begin{align*}	
(\alpha-\beta\delta^{(1)})(j)&=\alpha(j)-\beta((1+v)j(1+v)^{-1})\\
							   &=\alpha(j)-(1+w)\beta(j)(1+w)^{-1}\\
                               &=\alpha(j)-\beta(j)+w_2,	
\end{align*}
with $w_2\in J^2(A)$, so that $\alpha\sim_1 \beta\delta^{(1)}$.
One easily checks that $\delta^{(1)} \in [\id_{VQ}]_1$.

Treat $A$ as a $\Sigma$-bimodule via $\alpha$ ($=\beta\delta^{(1)}$ on $\Sigma$).
% Now we treat $A$ as a $\beta\delta^{(1)}(\Sigma)=\alpha\Sigma)$-bimodule. By abusing notation we say that $A$ is $\Sigma$-bimodule.
%Therefore we have that a map $\alpha-\beta \delta^{(1)}$, restricted to $V$ is a $\Sigma$-bimodule homomorphism.  
Denote by $t:J^2(A) \to J^2(k\db{VQ})$ a $\Sigma$-bimodule splitting of $\beta\delta^{(1)}$, so that $\beta\delta^{(1)} t = \id$ on $J^2(A)$.  Define $\delta^{(2)}_0$ to be the identity automorphism of $\Sigma$ and 
\begin{align*}
\delta_1^{(2)}: V & \to  T[[\Sigma,V]]    \\
j & \mapsto j+t(\alpha-\beta \delta^{(1)})(j),
\end{align*}
a $\Sigma$-bimodule homomorphism.  We obtain from $\delta_0^{(2)}, \delta_1^{(2)}$ (by Lemma \ref{Lemma UP tensor algebra}) a unique endomorphism of $k\db{VQ}$, which is an automorphism by construction.  One easily checks that $\delta^{(2)}\in [\id_{VQ}]_1$: 
\begin{align*}	
(\delta_1^{(2)}-\id_{VQ})(j)&=t(\alpha-\beta \delta^{(1)})(j)
					    =t(w_2)\in J^2(T[[\Sigma,V]]),	
\end{align*}
hence so is $\delta^{(1)}\delta^{(2)}$.  But $\alpha=\beta\delta^{(1)}\delta^{(2)}$ so that $\delta = \delta^{(1)}\delta^{(2)}$ is the required automorphism.
\end{proof}

\begin{remark}\label{counterexample to choose ideal on right}
Consider the algebra $k\db{\bullet\,\,\,\bullet} = k\times k$ and the algebra $A$ of lower triangular $2\times 2$-matrices.  The inclusions $\alpha,\beta:k\db{\bullet\,\,\,\bullet}\to A$
\[\alpha((1,0)) = \begin{pmatrix}
1 & 0 \\ 0 & 0
\end{pmatrix}, \quad \alpha((0,1)) = \begin{pmatrix}
0 & 0 \\ 0 & 1
\end{pmatrix},\]
\[\beta((1,0)) = \begin{pmatrix}
1 & 0 \\ 1 & 0
\end{pmatrix}, \quad \beta((0,1)) = \begin{pmatrix}
0 & 0 \\ -1 & 1
\end{pmatrix}\]
show that Proposition \ref{tilde is controlled on VQuiver side} is false for non-surjective homomorphisms.
\end{remark}

% \begin{remark} \label{tilde is controlled on VQuiver side2}
% A similar argument shows that if $\alpha,\beta:A\to k\db{VQ}$ are two injective morphisms in $\cat{PAlg}$, then $\alpha \sim_1 \beta$ if, and only if, there is $\delta \in [\id_{VQ}]_1$ such that $\alpha=\delta\beta$.
% % Mention that similarly to the proof above one proves the following. If $\alpha,\beta:A\to k\db{VQ}$ are two morphisms in $\cat{SPAlg}$, then $\alpha \sim_1 \beta$ is equivalent to existing of $\delta \in [\id_{VQ}]_1$ such that $\alpha=\delta\beta$.
% \end{remark}

\subsection{VQuivers with ideals and equivalence of categories}
Let $\cat{VQuiv}_{\infty}$ be the category with objects pairs $(VQ,[I])$, where $VQ$ is a finite Vquiver and $[I]$ is the $[\id_{VQ}]_1$ orbit of a relation ideal $I$ of $k[[VQ]]$.  That is,
\[[I]=\{\delta(I)\ | \ \delta\in [\id_{VQ}]_1 \}.\]
 A morphism 
\[\rho:(VQ,[I]) \to (VR,[K])\]
is a Vquiver map $\rho:VQ\to VR$ such that $k[[\rho]]$ takes $I'$ into $K'$ for some $I'\in [I]$ and $K'\in [K]$. 

Define the functor 
$\tn{GQ}_{\infty}(-):\cat{PAlg}_1 \to \cat{VQuiv}_\infty$ on objects by 
\[
	\tn{GQ}_{\infty}(A):=(\widetilde{\tn{GQ}}(A),[K_A]),
\]
where $K_A$ is the kernel of any representative $\widetilde{k}\db{\widetilde{\tn{GQ}}(A)}\to  A$ of the equivalence class of the counit of the adjunction of Theorem \ref{Theorem Adjunction Level 1}.
One checks using Proposition \ref{tilde is controlled on VQuiver side} that this definition does not depend on the choice of representative. 

On morphisms, we simply define 
\[\tn{GQ}_{\infty}([\alpha])=\widetilde{\tn{GQ}}([\alpha]).\]
One easily checks that these definitions yield a functor $\tn{GQ}_{\infty}(-)$.  We require a technical lemma.  Given an ideal $I$ of $k\db{VQ}$, denote by $\pi_I:k\db{VQ}\to k\db{VQ}/I$ the canonical projection.  

\begin{lemma}\label{Lemma unique hom gammaIIprime}
Fix an object $(VQ,[I])$ of $\cat{VQuiv}_{\infty}$ and two ideals $I,I'\in [I]$.  There is a unique isomorphism $\gamma_{II'}:\widetilde{k}\db{VQ}/I\to \widetilde{k}\db{VQ}/I'$ in $\cat{PAlg}_1$ such that $\gamma_{II'}[\pi_I] = [\pi_{I'}]$.
\end{lemma}

\begin{proof}
We work in $\cat{PAlg}$.  Let $\alpha\in [\tn{id}_{VQ}]_1$ be such that $\alpha(I) = I'$.  This map yields a well-defined algebra homomorphism $\overline{\alpha}:k\db{VQ}/I\to k\db{VQ}/I'$.  By definition, $\overline{\alpha}\pi_I = \pi_{I'}\alpha$.  Letting $\gamma_{II'} = [\overline{\alpha}]$, we obtain $\gamma_{II'}[\pi_I] = [\pi_{I'}]$ by Proposition \ref{tilde is controlled on VQuiver side}.  If $\delta$ is another morphism completing the triangle in $\cat{SPAlg}_1$, then $\gamma_{II'}[\pi_1] = \delta[\pi_1]$, so that $\gamma_{II'}=\delta$ by Lemma \ref{CancellationProp}.
\end{proof}

Throughout Section \ref{section adjoint equivalence}, $\pi_I$ denotes the canonical projection as above and $\gamma_{II'}$ denotes the isomorphism of Lemma \ref{Lemma unique hom gammaIIprime}.

\begin{theorem}
The functor $\tn{GQ}_{\infty}(-):\cat{PAlg}_1\to \cat{VQuiv}_{\infty}$ is an equivalence of categories. 
\end{theorem}

\begin{proof}
The functor $\tn{GQ}_{\infty}(-)$ is obviously essentially surjective. As $\widetilde{\tn{GQ}}(-)$ is faithful (see Remark \ref{remGQfaithful}), $\tn{GQ}_{\infty}(-)$ is as well.
Let $A, B$ be objects of $\cat{PAlg}_1$, which we may write as 
$\widetilde{k}\db{VQ}/I$ and 
$\widetilde{k}\db{VR}/K$, respectively. Let $\rho:(VQ,[I])\to (VR,[K])$ be a morphism  and suppose that $k\db{\rho}$ maps $I' \in [I]$ into $K'\in [K]$.  One may check that 
$\tn{GQ}_{\infty}({\gamma_{KK'}}^{-1} \widetilde{k}[[\rho]] \gamma_{II'}) = \rho$, so that $\tn{GQ}_{\infty}(-)$ is full and hence yields an equivalence of categories.
\end{proof}

\subsection{The functor $k_{\infty}\db{-}$ and adjoint equivalence}

It is not difficult to construct (in the spirit of Remark \ref{counterexample to choose ideal on right}) a morphism $\rho:(VQ,[I])\to (VR,[K])$ in $\cat{VQuiv}_{\infty}$ having the property that there are representatives $I'\in [I]$ and $K'\in [K]$ such that $I'$ does not land inside any representative of $[K]$ under $\rho$ and nor does any representative of $[I]$ land inside $K'$.  This makes the construction of an explicit left adjoint to $\tn{GQ}_{\infty}(-)$ rather tricky.  However, Proposition \ref{tilde is controlled on VQuiver side} allows us to be more precise when we restrict to the important special case where we have surjective maps on both sides.  Denote by $\cat{SVQuiv}_{\infty}$ the subcategory of $\cat{VQuiv}_{\infty}$ whose morphisms are surjective maps of Vquivers.

\begin{lemma} \label{any ideal is preimage of some ideal}
If $\rho:(VQ,[I])\to (VR,[K])$ is a morphism in $\cat{SVQuiv}_{\infty}$, then for any $K'' \in [K]$, there is $I''\in [I]$ such that $k[[\rho]]$ maps $I''$ into $K''$.
\end{lemma}
\begin{proof}
Suppose that $k\db{\rho}$ takes $I'\in [I]$ into $K'\in [K]$ and fix some $K''\in [K]$. By definition, there exists $\delta \in [\id_{VR}]_1$ such that $\delta(K')=K''$. 
Because $\delta \sim_1 \id_{VR}$, we have 
$\delta k[[\rho]] \sim_1 k[[\rho]]$.  Applying Proposition \ref{tilde is controlled on VQuiver side}, there exists $\delta' \in [\id_{VQ}]_1$ such that
$\delta k[[\rho]] \delta'=k[[\rho]]$. 
The automorphism $\delta'^{-1}$ maps the ideal $I'$ into some ideal $I''\in [I]$ and we have
\[
	k[[\rho]](I'')=(\delta k[[\rho]] \delta')(I'')=(\delta k[[\rho]])(I')\subseteq \delta(K')=K''.
\]
\end{proof}

We define the functor $k_{\infty}[[-]]:\cat{SVQuiv}_\infty \to \cat{SPAlg}_1$  as follows.  For each object $(VQ,[I])$, we choose a representative $I\in [I]$ and define
\[
	k_{\infty}\db{(VQ,[I])} := k[[VQ]]/I.
\]
Given a morphism $\rho:(VQ,[I])\to (VR,[K])$ in $\cat{SVQuiv}_{\infty}$ sending $I'\in [I]$ into $K'\in [K]$, we define
\[
	k_{\infty}\db{\rho} := {\gamma_{KK'}}^{-1}\widetilde{k}\db{\rho}\gamma_{II'}: k\db{VQ}/I\to k\db{VR}/K,
\]
where $\gamma_{KK'},\gamma_{II'}$ are defined as in Lemma \ref{Lemma unique hom gammaIIprime}.

\begin{lemma}
With the definitions above, $k_{\infty}[[-]]$ is a well-defined functor. 
\end{lemma}

\begin{proof}
First we show that $k_{\infty}[[\rho]]$ does not depend on the choice of representatives $I',K'$.  Suppose that $I''\in [I], K''\in [K]$ are other representatives such that $k\db{\rho}(I'')\subseteq K''$.  Define $\gamma_{I'I''} = \gamma_{II''}{\gamma_{II'}}^{-1}$ and $\gamma_{K'K''} = \gamma_{KK''}{\gamma_{KK'}}^{-1}$.  In the following diagram, each small shape commutes by definition except the inner triangles, which are easily seen to commute.

\[
\xymatrix{    
&\frac{k\db{VQ}}{I'}\ar[rrr]\ar[dd]_{\gamma_{I'I''}}  &&& \frac{k\db{VR}}{K'}\ar[dd]^{\gamma_{K'K''}}\ar[dr]^{{\gamma_{KK'}}^{-1}}   \\
\frac{k\db{VQ}}{I}\ar[dr]_{\gamma_{II''}}\ar[ur]^{\gamma_{II'}}&& k\db{VQ}\ar[r]_{\widetilde{k}\db{\rho}}\ar[ul]_{\pi_{I'}}\ar[dl]_{\pi_{I''}} & k\db{VR}\ar[ur]_{\pi_{K'}}\ar[dr]_{\pi_{K''}} && \frac{k\db{VR}}{K} \\
&\frac{k\db{VQ}}{I''}\ar[rrr]  &&&  \frac{k\db{VR}}{K''}\ar[ur]_{{\gamma_{KK'}}^{-1}}    }
\]
It follows that the large hexagon commutes.  But the upper path is $k_{\infty}\db{\rho}$ defined with respect to $I',K'$ and the lower is $k_{\infty}\db{\rho}$ defined with respect to $I'',K''$.  Hence $k_{\infty}\db{\rho}$ is well-defined.

\medskip

Now we check functoriality.  We have
\[k_{\infty}\db{\tn{id}_{(VQ,I)}} = {\gamma_{II}}^{-1}\tn{id}_{kVQ/I}\gamma_{II} = \tn{id}_{kVQ/I}.\]
Suppose we have
\[(VQ,[I])\xrightarrow{\rho} (VR,[K]) \xrightarrow{\theta} (VS,[L])\]
and suppose that $k\db{\theta}(K')\subseteq L'$.  By Lemma \ref{any ideal is preimage of some ideal} we can find $I'$ such that $k\db{\rho}(I')\subseteq K'$, and so we have
\begin{align*}
k_{\infty}\db{\theta}{k}_{\infty}\db{\rho} & = ({\gamma_{LL'}}^{-1}\widetilde{k}\db{\theta}\gamma_{KK'})({\gamma_{KK'}}^{-1}\widetilde{k}\db{\rho}\gamma_{II'}) \\
& = {\gamma_{LL'}}^{-1}\widetilde{k}\db{\theta}k\db{\rho}\gamma_{II''} \\
& = {\gamma_{LL'}}^{-1}\widetilde{k}\db{\theta\rho}\gamma_{II''} \\
& = k_{\infty}\db{\theta\rho}.
\end{align*}
\end{proof}

\begin{lemma}\label{Lemma adjunction factorization property}
A morphism $\gamma:\widetilde{k}\db{VQ}\to A$ factors as $\gamma=\varepsilon_A\circ \widetilde{k}\db{\Phi_{VQ,A}(\gamma)}$, where $\Phi_{VQ,A}$ is the inverse of $\Psi_{VQ,A}$ as in Section \ref{section main adjunction} and $\varepsilon$ is the counit of the adjuntion of Theorem \ref{Theorem Adjunction Level 1}.

\end{lemma}
\begin{proof}
This is standard abstract nonsense.   Letting $\eta$ be the unit of the adjunction, we have that $\Phi_{VQ,A}(\gamma) = \widetilde{\tn{GQ}}(\gamma)\eta_{VQ}$.  Hence
\begin{align*}
\varepsilon_A\circ \widetilde{k}\db{\Phi_{VQ,A}(\gamma)} & = \varepsilon_A \circ\widetilde{k}\db{\widetilde{\tn{GQ}}(\gamma)\eta_{VQ}} \\
& = \varepsilon_A\circ \widetilde{k}\db{\widetilde{\tn{GQ}}(\gamma)}\circ \widetilde{k}\db{\eta_{VQ}} \\
& =  \gamma\circ\varepsilon_{\widetilde{k}\db{VQ}}\circ \widetilde{k}\db{\eta_{VQ}} \qquad (\hbox{by naturality of }\varepsilon)\\
& =  \gamma. \qquad\qquad\qquad\qquad (\hbox{by counit-unit equations})
\end{align*}
\end{proof}

We have the following 

\begin{theorem}
The functors $k_{\infty}[[-]]$ and $\tn{GQ}_{\infty}(-)$ form an adjoint equivalence between the categories $\cat{SVQuiv}_\infty$ and $\cat{SPAlg}_1$.
\end{theorem}

\begin{proof}
Denote the mutually inverse isomorphisms of Adjuntion \ref{Theorem Adjunction Level 1} by
\[\Phi_{VQ,A} :  \Hom_{\cat{SPAlg}_1}(\widetilde{k}\db{VQ},A) \longleftrightarrow \Hom_\cat{SVQuiv}(VQ,\widetilde{\tn{GQ}}(A)) : \Psi_{VQ,A}\]
and let $\varepsilon$ be the counit of this adjunction.  Let $\pi_I:k\db{VQ}\to k\db{VQ}/I$ be the canonical projection.  We define the map
\[\Phi^{\infty}_{(VQ,[I]),A} : \Hom_{\cat{SPAlg}_1}(k\db{VQ}/I, A) \to \Hom_{\cat{SVQuiv}_{\infty}}((VQ,[I]),\tn{GQ}_{\infty}(A))\]
by $\alpha\mapsto \Phi_{VQ,A}(\alpha\pi_I)$.  This map is
\begin{itemize}
\item Well-defined by Lemma \ref{Lemma adjunction factorization property}, because $\alpha\pi_I = \varepsilon_A\circ\widetilde{k}\db{\Phi_{VQ,A}(\alpha\pi_I)}$, thus $\widetilde{k}\db{\Phi_{VQ,A}(\alpha\pi_I)}(I)\subseteq \tn{Ker}(\varepsilon_A)$ as required.

\item Injective, because $\alpha\mapsto \alpha\pi_I$ is injective by Lemma \ref{CancellationProp} and $\Phi_{VQ,A}$ is injective.

\item Surjective: fix $\rho:(VQ,[I])\to (\widetilde{\tn{GQ}}(A),[K])$ (where $K$ the kernel of a representative of $\varepsilon_A$).  Then $\Psi_{VQ,A}(\rho):\widetilde{k}\db{VQ}\to A$ has some $I$ in its kernel (again by Lemma \ref{Lemma adjunction factorization property}) and so can be written as $\alpha\pi_I$ for some $\alpha:k\db{VQ}/I\to A$.  Now
\[\Phi^{\infty}_{(VQ,[I]),A}(\alpha) = \Phi_{VQ,A}\Psi_{VQ,A}(\rho) = \rho.\]

\item Natural in $(VQ,[I])$: fix $\rho:(VQ,[I])\to (VR,[L])$ and suppose $\rho(I)\subseteq L$.  Naturality can be seen  by applying the  functor $\tn{Hom}_{\cat{SPAlg}_1}(-,A)$  to the commutative square $k^{\infty}\db{\rho}\pi_I = \pi_L\widetilde{k}\db{\rho}$ and gluing it to the corresponding square for $\Phi$.

\item Natural in $A$: this follows by a similar argument.

\end{itemize}
It follows from the above observations that the functors are adjoint. The component $\eta_{(VQ,[I])}^{\infty}$  of the unit at $(VQ,[I]) \in \cat{SVQuiv}_\infty$ is
\[
\Phi_{VQ,k[[VQ]]/I}(\pi_I),
\]
which is invertible by Lemma \ref{unit is isomorphism}.
The component $\varepsilon_A^{\infty}$ of the counit at $A\in \cat{SPAlg}_1$ is the factorization of $\varepsilon_A$ through its kernel $K$:
\[
	\varepsilon_A=\varepsilon_A^{\infty}\pi_K,
\]
and so $\varepsilon_A^{\infty}$ is an isomorphism.
\end{proof}

\end{document}